\documentclass[a4paper,reqno,10pt]{amsart}
\usepackage[margin=2cm,papersize={18.25cm,26cm}]{geometry}

\newcommand{\CTP}{CTP}

\usepackage{mathrsfs}
\usepackage{amsfonts}
\usepackage[centertags]{amsmath}

\usepackage[mathcal]{euscript}
\usepackage{latexsym,amsmath,amssymb}
\usepackage{graphicx,color}
\usepackage[active]{srcltx} 

\newcommand{\ds}{\displaystyle}
\newcommand{\N}{\mathbb{N}}
\newcommand{\R}{\mathbb{R}}
\newcommand{\K}{\mathscr{K}}

\newcommand{\eps}{\epsilon}
\newcommand{\gref}{\gamma_{\mathrm{ref}}}
\newcommand{\dgref}{{\dot\gamma}_{\mathrm{ref}}}

\usepackage{eqnarray}

\newtheorem{thm}{Theorem}[section]
\newtheorem{prop}[thm]{Proposition}
\newtheorem{defn}[thm]{Definition}
\newtheorem{lemma}[thm]{Lemma}
\newtheorem{corol}[thm]{Corollary}

\newtheorem{remarkth}[thm]{Remark}

\newenvironment{remark}{\begin{remarkth}\upshape}{\hfill$\diamond$\end{remarkth}}

\renewcommand{\emph}[1]{{\bfseries\itshape{#1}}}

\def\qed{\ifvmode\removelastskip\fi
{\unskip\nobreak\hfil\penalty50\hbox{}\nobreak\hfil \hbox{\vrule
height1.2ex width1.2ex}\parfillskip=0pt \finalhyphendemerits=0
\par \smallskip}}

\title{New high order sufficient conditions for configuration tracking}
\author{{\sc M. Barbero-Li\~n\'an}, \
 {\sc M. Sigalotti}
 }
\begin{document}

\maketitle

\begin{abstract}
In this paper, we propose new conditions guaranteeing that the trajectories of a mechanical
control system
 can track any curve on the configuration
manifold. We focus on systems that can be represented as forced
affine connection control systems and we generalize the sufficient
conditions for
tracking known in the literature.
The new results are proved by a combination of averaging procedures by highly oscillating controls with the notion of kinematic reduction.  
\end{abstract}


\section{Introduction}\label{S1Intro}

New geometric techniques are used to generalize tracking conditions known in the literature~\cite{2005BulloAndrewBook,BS2010,Mario}. The tracking problem plays a key role in the performance of robots and mechanical systems such as submarines and hovercrafts in order to avoid obstacles, stay nearby a preplanned trajectory, etc. 

Mechanical control systems are control-affine systems on the tangent bundle of the con\-fi\-gu\-ration manifold $Q$. 
In order to simplify the motion planning tasks for these control systems, a useful tool has been introduced  in the geometric control literature, namely, the notion of kinematic reduction. Such a procedure consists in identifying 
a control-linear system on $Q$ whose trajectory mimic those of the mechanical system. This approach has been useful to describe controllability, planning properties~\cite{2005BulloAndrewBook} and optimality~\cite{2008Texas} of mechanical systems. However, as described in~\cite{2005BulloAndrewBook}, kinematic reduction is not always possible, some conditions related to the symmetric closure of the control vector fields of both systems under study must be satisfied.

In our previous work~\cite{BS2010} we extended the first-order sufficient conditions for tracking proposed in~\cite{2005BulloAndrewBook} by using different families of vector fields, possibly of infinite cardinality. Related constructions to generate admissible directions for tracking have been proposed in \cite{Bressan,Jorge} (see also \cite{AgSa1,AgSa2}). 

 Our first goal in this current paper is to establish a relationship between families of vector fields defined pointwise and sets of sections of the tangent bundle defined in a recurrent way, similarly to the classical Malgrange theorem~\cite{Malgrange67}.
This new pointwise characterization of families of vector fields used in~\cite{BS2010} allows to use kinematic reduction in order to obtain more general sufficient tracking conditions. As a result, it can be proved that an underwater vehicle with a natural choice of control vector fields is always trackable, even in the most symmetric case (see the example in Section~\ref{SExample} for more details).

The paper is organized as follows. Section~\ref{S2Notation} contains all the necessary background in forced affine connection control systems~\cite{2005BulloAndrewBook}. Section~\ref{STrack} defines the notion of trackability under study. 
After recalling in Section 4.1 the high-order tracking conditions known in the literature and obtained by averaging theory~\cite{BS2010}, kinematic reduction is used to obtain more general tracking conditions in Section~\ref{SNewTrack}. The full characterization of the trackability of the system 
describing the motion of an underwater vehicle is achieved in Section~\ref{SExample}, concluding the study started in~\cite{Mario} and continued in~\cite{BS2010}.

\section{Notation and preliminaries}\label{S2Notation}

Denote by $\mathbb{N}$ the set of positive natural numbers
and write  $\N_0$ for $\N\cup\{0\}$.
Fix $n\in\mathbb{N}$.
From now on, $Q$ is a $n$\textendash{}dimensional smooth manifold
and $\mathfrak{X}(Q)$ denotes the set of smooth vector fields on
$Q$. All  vector fields are considered smooth as functions on $Q$, unless otherwise
stated. Let $\tau_Q\colon TQ \rightarrow Q$ be the canonical tangent
projection. A \textit{vector field 
along
$\tau_Q$} is a mapping $X\colon TQ \rightarrow TQ$ such that
$\tau_Q\circ X= \tau_Q$.
We denote by $I$ a compact interval of the type $[0,\tau]$, $\tau>0$.

\subsection{Affine connection control systems}

The trajectories $\gamma:I\to Q$ of a Lagrangian mechanical systems
on a manifold $Q$ are minimizers of the action functional
\[A_L(\gamma)=\int_I L(t,\dot{\gamma}(t)){\rm d}t\]
associated with a Lagrangian function $L\colon \mathbb{R} \times TQ\rightarrow \mathbb{R}$.

 The solutions to this variational problem 
must 
satisfy
the well-known Euler\textendash{}Lagrange
equations,
\begin{equation}\label{euler-lagrange}
\frac{{\rm d}}{{\rm d}t}\left( \frac{\partial L}{\partial
v^i}\right)-\frac{\partial L}{\partial q^i}=0,\quad i=1,\dots,n,
\end{equation}
where $(q^i,v^i)$ are local coordinates for $TQ$. Here we consider
controlled Euler\textendash{}Lagrange equations obtained by
modifying the right-hand side on the above equation, as
follows:
\[\frac{{\rm d}}{{\rm d}t}\left( \frac{\partial L}{\partial
v^i}\right)-\frac{\partial L}{\partial
q^i}=\sum_{a=1}^ku_aY_a^i,\quad i=1,\dots,n,\] with $u_a\colon I
\rightarrow \mathbb{R}$, $Y_a^i\colon Q \rightarrow \mathbb{R}$.

 When the manifold $Q$ is endowed with the Riemannian structure given by a Riemannian metric
$g$ and the Lagrangian function $L_g(v_q)=\frac{1}{2}g(v_q,v_q)$ is
considered, the solutions to (\ref{euler-lagrange}) turn out to be the
geodesics of the Levi\textendash{}Civita affine connection
$\nabla^g$ associated with the Riemannian metric. (See
\cite{2005BulloAndrewBook} for more details and
for many examples of mechanical control systems that fit in this
description.)

When control forces are added to the geodesic equations we obtain an
affine connection control system
\[\nabla^g_{\dot{\gamma}(t)}\dot{\gamma}(t)=\sum_{a=1}^ku_a(t)Y_a(\gamma(t)),\]
where $Y_1,\dots,Y_k$ are vector fields on $Q$.

The notion of affine connection control system can be extended
without the need of the Levi\textendash{}Civita connection.

\begin{defn} An \textbf{affine connection}
is a mapping
\[\begin{array}{rcl}
\nabla\colon\mathfrak{X}(Q)\times \mathfrak{X}(Q)&
\longrightarrow & \mathfrak{X}(Q)\\
(X,Y) & \longmapsto & \nabla(X,Y)=\nabla_XY,
\end{array}\] satisfying the following
properties:
\begin{enumerate}
\item $\nabla$ is $\mathbb{R}$\textendash{}linear on $X$ and on $Y$;
\item $\nabla_{fX}Y=f\nabla_XY$ for every $f\in {\mathcal C}^{\infty}(Q)$;
\item $\nabla_XfY=f\nabla_XY+\left(Xf \right)Y$, for every $f\in {\mathcal C}^{\infty}(Q)$. (Here $Xf$ denotes the derivative
of $f$ in the direction $X$.)
\end{enumerate}
\end{defn}

 The mapping $\nabla_XY$ is called the \textit{covariant derivative of $Y$
 with respect to $X$}. Given local coordinates $(q^i)$ on $Q$, the
 \textit{Christoffel symbols for the affine connection} in these
 coordinates are given by
 \[\ds{\nabla_{\frac{\partial}{\partial q^j}}\frac{\partial}{\partial q^r}=\sum_{i=1}^n
 \Gamma^i_{jr}\frac{\partial}{\partial q^i},\qquad j,r=1,\dots,n.}\]
 From the properties of the affine connection, we have
 \[\ds{\nabla_XY=\sum_{i,j,r=1}^n\left(X^j\frac{\partial Y^i}{\partial q^j}+\Gamma^i_{jr}X^jY^r
 \right)\frac{\partial}{\partial q^i}},\]
 where $X=\sum_{i=1}^nX^i\partial/\partial q^i$ and $Y=\sum_{i=1}^nY^i\partial/\partial
 q^i$.

\begin{defn}\label{Def-FACCS} A \textbf{forced affine connection control system
(FACCS)} is a control mechanical system given by $\Sigma=(Q,\nabla,Y, \mathscr{Y},U)$ where
\begin{itemize}
\item $Q$ is
a smooth $n$\textendash{}dimensional manifold called the
\textit{configuration manifold},
\item $Y$ is a smooth time-dependent vector field along the projection $\tau_Q\colon TQ\rightarrow Q$, 
 affine with respect to the velocities, 
\item $\mathscr{Y}$ is a set of $k$
control vector fields on $Q$, and
\item $U$ is a measurable subset of $\R^k$.
\end{itemize}
A trajectory $\gamma\colon I\subset \mathbb{R} \rightarrow Q$ is
\textbf{admissible for $\mathbf{\Sigma}$} if $\dot \gamma\colon I
\rightarrow TQ$ is absolutely continuous and there exists a
measurable and bounded control $u\colon I \rightarrow U$ such that
the dynamical equations of the control system~$\Sigma$
\begin{equation}\label{nabla}\nabla_{\dot{\gamma}(t)}\dot{\gamma}(t)=Y(t,\dot{\gamma}(t))+\sum_{a=1}^ku_a(t)Y_a(\gamma(t)),
\end{equation}
are fulfilled (for almost every $t\in I$).
\end{defn}

The vector field $Y$ 
includes all the non-controlled external forces; e.g., the potential
and the non-potential forces. The assumption that $Y$
is affine with respect to the velocities means that, in every
local system of coordinates $(q^i,v^i)$ on $TQ$, $Y$ can be written
as
$$Y(t,v_q)=Y_0(t,q)+\sum_{i=1}^n v^i Y^i(t,q).$$

Equation $(\ref{nabla})$ can be rewritten as a
first-order control-affine system on $TQ$,
\begin{equation}
\label{first}
\dot{\Upsilon}(t)=Z(\Upsilon(t))+Y^V(t,\Upsilon(t))+\sum_{a=1}^ku_a(t)Y^V_a(\Upsilon(t)),
\end{equation}
where $\Upsilon\colon I \rightarrow TQ$ is such that  $\tau_Q \circ
\Upsilon=\gamma$, $Z$ is the geodesic spray associated with the affine
connection on $Q$ and,
for every $X\in \mathfrak{X}(Q)$,  $X^V$ denotes the vertical lift of 
$X$ (see \cite{AbrahamMarsden} for more details).

Apart from the usual Lie bracket that provides $\mathfrak{X}(Q)$
with a Lie algebra structure, one can associate with 
$\nabla$ the following product in $\mathfrak{X}(Q)$. 
\begin{defn} The \textbf{symmetric product} is the map
\begin{eqnarray*}\label{sym} \langle \cdot \colon \cdot \rangle\colon
\mathfrak{X}(Q) \times  \mathfrak{X}(Q) & \longrightarrow &
\mathfrak{X}(Q) \\
(X,Y) & \longmapsto & \nabla_XY +\nabla_Y X.
\end{eqnarray*}
\end{defn}

It can be proved that
\begin{equation}\label{triple_bracket}
[Y_a^V,[Z,Y_b^V]]=\langle Y_a \colon Y_b
\rangle^V
\end{equation}
(see \cite{2005BulloAndrewBook}).

\section{Tracking problem}\label{STrack}

We consider here the problem arising when one tries to follow a
particular trajectory on the configuration manifold, called \textit{reference} or
\textit{target} trajectory, which is in general not a solution of the FACCS
considered. A trajectory is successfully tracked if there exist
solutions to the FACCS that approximate it arbitrarily well.

Consider any distance ${\rm d}\colon Q\times Q \rightarrow \mathbb{R}$ on $Q$
whose corresponding metric topology coincides with the topology on $Q$.

\begin{defn} A curve $\gamma\colon I\rightarrow Q$
of class ${\mathcal C}^1$
is
\textbf{trackable for the FACCS $\Sigma$} if, for every strictly
positive tolerance $\epsilon$, there exist a control $u^{\epsilon}\in L^\infty(I,U)$
and a 
solution $\xi^{\epsilon}\colon I
\rightarrow Q$  to $\Sigma$ corresponding to  $u^\eps$ 
such that 
${\xi}^\epsilon(0)={\gamma}(0)$ and
\[{\rm d} (\gamma(t), \xi^{\epsilon}(t))<\epsilon\]
for every $t\in I$. The trajectory 
is said to be \textbf{strongly trackable for $\Sigma$} if, in addition to the above requirements, for every $\eps>0$ the approximating trajectory ${\xi}^\epsilon$ 
may be found also satisfying 
$\dot{\xi}^\epsilon(0)=\dot{\gamma}(0)$. 

A control system $\Sigma$ satisfies the \textbf{configuration tracking property (CTP)} (respectively, the \textbf{strong configuration tracking property (SCTP)}) if every curve on
$Q$ of class $\mathcal{C}^1$  is trackable (respectively, strongly trackable)  for $\Sigma$.
\end{defn}
\begin{remark}
Since any $\mathcal{C}^1$ curve can be uniformly approximated, with arbitrary precision, by a smooth curve having the same tangent vector at its initial point,
then $\Sigma$ satisfies the  \CTP\ (respectively, the SCTP) if and only if every curve on
$Q$ of class $\mathcal{C}^\infty$  is  trackable (respectively, strongly trackable) for $\Sigma$.
\end{remark}

\subsection{Tracking results for control-linear systems}\label{control-linear}

A {\it control-linear system} (also  called {\it driftless kinematic system}) on $Q$ is a triple $(Q,\mathscr{X},U)$ where 
 $\mathscr{X}$ is a finite subset $\{X_1,\dots,X_m\}$ of $\mathfrak{X}(Q)$ and $U$ is a measurable subset of $\R^m$, identified with the  
control system 
$$\dot \gamma(t)=\sum_{a=1}^m u_a(t) X_a(\gamma(t)),\qquad \gamma(t)\in Q,$$
where $u_1(\cdot),\dots,u_m(\cdot)$ are measurable and bounded functions with $(u_1(t),\dots,u_m(t))\in U$ for every $t$.

\begin{prop}[See \cite{Liu,LS}]\label{LiuSussmann}
Let $X_1,\dots,X_m$ be smooth vector fields on $Q$ and 
take $\kappa\in \N$. 
Let $\{X_1,\dots, X_{\hat m}\}$ be the set of all  Lie brackets of the vector fields $X_1,\dots,X_m$ of length less than or equal to $\kappa$. 
Assume that $\gamma:I\to Q$ is a $\mathcal{C}^\infty$ curve such that
$$\dot \gamma(t)=\sum_{a=1}^{\hat m} w_a(t) X_a(\gamma(t)),$$
with $w:I\to \R^{\hat m}$ smooth.
Then, for  every $\eps>0$ there exists a 
solution $\gamma_\eps$ of the control-linear system $(Q,\{X_1,\dots,X_m\},\R^m)$
with smooth control $u_\eps:I\to \R^m$  and 
initial condition 
$\gamma_\eps(0)=\gamma(0)$ 
such that $d(\gamma(t),\gamma_\eps(t)) <\eps$ for
every $t\in I$. 
\end{prop}

From the above proposition we deduce the following result. (Similar arguments can be found in \cite{Jakubczyk}.)

\begin{corol}\label{tracking-control-linear}
If the Lie algebra $\mathrm{Lie}(X_1,\dots,X_m)$
generated by $X_1,\dots,X_m$ has constant rank on $Q$, then for every smooth curve 
$\gamma:I\to Q$ and for  every $\eps>0$ there exists a 
solution $\gamma_\eps$ of the control-linear system $(Q,\{X_1,\dots,X_m\},\R^m)$
with smooth control $u_\eps:I\to \R^m$  and 
initial condition 
$\gamma_\eps(0)=\gamma(0)$ 
such that $d(\gamma(t),\gamma_\eps(t)) <\eps$ for
every $t\in I$. 
\end{corol}
\begin{proof}
The proof works by covering the 
compact set $\gamma(I)$ by finitely many open sets $\Omega_1,\dots,\Omega_K$ 
of $Q$  such that for every $j=1,\dots,K$ there exists on $\Omega_j$ a basis of the distribution
 $\mathrm{Lie}(X_1,\dots,X_m)$ made of Lie brackets of $X_1,\dots,X_m$.
Let $\kappa$ be the maximum of the length of the brackets used to construct such bases
and let $\{X_1,\dots, X_{\hat m}\}$ be the set of all  Lie brackets of the vector fields $X_1,\dots,X_m$ of length less than or equal to $\kappa$. 
Then 
$$\dot \gamma(t)=\sum_{a=1}^{\hat m} w_a(t) X_a(\gamma(t)),$$
with $w:I\to \R^{\hat m}$ smooth and we conclude by  Proposition~\ref{LiuSussmann}.
\end{proof}

\subsection{Previous strong configuration tracking results}\label{SPreResults}

Conditions guaranteeing the  SCTP have been obtained in \cite{BS2010}, generalizing previous results presented in \cite{2005BulloAndrewBook} (in particular Theorem 12.26) and in \cite{Mario}. 
We recall them here below in a version adapted to what follows.
The main difference of these statements from the ones of Theorem~4.4 and Corollary~4.7 in \cite{BS2010} is that here we focus on the strong configuration trackability of a given 
trajectory, instead of looking at the SCTP. The proof is however exactly the same, since the proof proposed in \cite{BS2010} is based on an argument where the target  trajectory is fixed.

\begin{prop}\label{thm:oldtrackK}
Let $\Sigma=(Q,\nabla,Y,\mathscr{Y},\mathbb{R}^{k})$ be  a FACCS.
Construct the following
set of vector fields on $Q$:
\begin{equation}\label{setK}\begin{array}{rcl} {\mathscr{K}}_0&=&
\overline{{\rm span}_{\mathcal{C}^\infty(Q)}\mathscr{Y}},\\[1.5mm]
\mathscr{K}_l&=&\overline{\mathscr{K}_{l-1}- {\rm co}\left\{\langle
Z \colon Z \rangle \mid Z\in {\rm L}(\mathscr{K}_{l-1})\right\}},
\end{array}\end{equation} for $l\in \mathbb{N}$, where, for $A\subset \mathfrak{X}(Q)$,
 ${\rm L}(A)=A\cap (-A)$, ${\rm co}(A)$ denotes the convex hull  of $A$, and
 $\overline{A}$ is the closure of $A$
in $\mathfrak{X}(Q)$ with respect to the topology of
the uniform convergence on compact sets.

Fix a smooth reference trajectory $\gref:I\rightarrow Q$ of class
${\mathcal C}^\infty$. 
Assume that there exist $l,N\in
\mathbb{N}$, $\lambda_1,\dots,\lambda_N\in C^\infty(I,[0,+\infty))$,  and $Z_1,\dots,Z_N\in \mathscr{K}_l$ such that 
$$\nabla_{\dgref(t)}\dgref(t)-Y(t,\dgref(t))=\sum_{a=1}^N \lambda_a(t)Z_a(\gref(t)),\qquad \forall t\in I\,. $$
Then $\gref$
is strongly trackable. 
\end{prop}

\begin{prop}\label{prop:track-Z}
Let $\Sigma=(Q,\nabla,Y,\mathscr{Y},\R^k)$ be a FACCS. Define the following sets of
vector fields for $l\in \mathbb{N}$,
\begin{align} \mathscr{Z}_0=& \mathscr{Y},\nonumber\\
\mathscr{Z}_l=& \mathscr{Z}_{l-1} \cup \{\langle Z_a\colon
Z_b\rangle \mid Z_a,Z_b\in \mathscr{Z}_{l-1}\}.\label{setZ}
\end{align}
Assume that there exists $l\in \N$ such that 
for
each $i\in\{0, \ldots, l-1\}$, for each $Z\in \mathscr{Z}_i$,
$\langle Z\colon Z\rangle\in {\rm
span}_{\mathcal{C}^\infty(Q)}\mathscr{Z}_i$.

Let $\mathscr{Z}_l=\{Z_1,\dots,Z_N\}$. 
Fix a smooth reference trajectory $\gref:I\rightarrow Q$ of class
${\mathcal C}^\infty$. 
If there exist  $\lambda_1,\dots,\lambda_N\in C^\infty(I,\R)$ such that 
$$\nabla_{\dgref(t)}\dgref(t)-Y(t,\dgref(t))=\sum_{a=1}^N \lambda_a(t)Z_a(\gref(t)),\qquad \forall t\in I, $$
 then $\gref$
is strongly trackable. 
\end{prop}

\section{A generalization of tracking conditions}\label{SMain}

Before introducing the new results about tracking, we need some technical lemmas described in Section~\ref{s-gregoire} and to define kinematic reduction in Section~\ref{s-kinred}. All that is necessary to prove the new theorem about trackability in Section~\ref{SNewTrack}.

\subsection{Pointwise and sectionwise characterisation of $\mathscr{K}_l$}\label{s-gregoire}

The results in this section, in the spirit of the classical Malgrange theorem (see \cite{Malgrange67}), aim at characterizing the  
sets $\mathscr{K}_l$ of sections of $TQ$,   introduced above, in terms of iterated computations of subsets of $TQ$. 
Let us then associate with a
family $\mathscr{Y}=\{Y_1,\dots, Y_k\}\subset \mathfrak{X}(Q)$, in addition 
to the family $\mathscr{K}_l$, the family of subsets of $TQ$ 
defined pointwise, for every $q\in Q$, as
\begin{align}
\widehat{{\mathscr{K}}}_{0,q}&=\;
{\rm span}_{\mathbb{R}}\mathscr{Y}(q), \label{eq:HatK0}\\[1.5mm]
\widehat{\mathscr{K}}_{l,q}&=\; \overline{\widehat{\mathscr{K}}_{l-1,q}- {\rm
co}\left\{\langle Z \colon Z \rangle(q) \mid Z\in \mathfrak{X}(Q),\;Z(q')\in {\rm
L}(\widehat{\mathscr{K}}_{l-1,q'}) \; \forall \; q'\in Q\right\}},
\label{eq:HatKl}
\end{align}
where for any $A\subset \mathfrak{X}(Q)$ we write
$A(q)= \{Y(q)\mid Y\in A\}\subseteq T_qQ$.

Recall that each $\mathscr{K}_l$ is a convex cone in $\mathfrak{X}(Q)$ for the $\mathcal{C}^0_{\mathrm{loc}}$ topology (see \cite[Proposition 4.1]{BS2010}). It is also clear that the recursive definition of $\widehat{{\mathscr{K}}}_{l,q}$ describes a closed convex cone of $T_q Q$.

We need a preliminary result to establish 
the equivalence between the two definitions. 

\begin{lemma} Let $\mathscr{H}$ be a ${\mathcal C}^0_{{\rm loc}}$-closed set of $\mathfrak{X}(Q)$ and assume that $\mathscr{H}$
is closed with respect to finite linear combinations with coefficients in $\mathcal{C}^\infty(Q,[0,\infty))$. 
Then
\begin{equation}
\mathscr{H}=\{V\in\mathfrak{X}(Q) \mid  V(q)\in \mathscr{H}(q) \quad \forall \;
q\in Q\}. \label{eq:lemmaH}
\end{equation}
\label{lemma:H}

\proof The inclusion
\begin{equation*}
\mathscr{H}\subseteq \{V\in\mathfrak{X}(Q) \mid  V(q)\in \mathscr{H}(q) \quad
\forall \; q\in Q\}
\end{equation*}
is trivial and we are left to prove the opposite one. 

If $V \in\mathfrak{X}(Q)$ and  $V(q)\in \mathscr{H}(q)$ for all $q\in Q$, then for every $q\in
Q$ there exists $W^q\in \mathscr{H}$ such that $W^q(q)=V(q)$. For
all $\epsilon >0$ there exists a closed neighbourhood
$\Omega^{q,\epsilon}$ such that \begin{equation*} \|
W^q-V\|_{\infty,\Omega^{q,\epsilon}} \leq \epsilon,\end{equation*} where
$\| \cdot \|_{\infty,\Omega^{q,\epsilon}}$ is the supremum norm
restricted to $\Omega^{q,\epsilon}$, with respect to any fixed Riemannian structure on $Q$. For every $\epsilon>0$, 
$\{\Omega^{q,\epsilon}\}_{q\in Q}$ is an open covering of $Q$.
Let $(Q_n)_{n\in \mathbb{N}}$ be an increasing sequence of compact subsets whose interiors cover $Q$. 
For every $n\in\mathbb{N}$ there exists a finite covering $\Omega^{q_1,1/n},\dots,\Omega^{q_{r_n},1/n}$ of $Q_n$ and  a partition of unity  $a_1,\dots,a_{r_n}$
subordinated to $\{\Omega^{q_i,1/n}\}_{i=1}^{r_n}$ 
such that
\begin{equation*} \left\| \sum_{i=1}^{r_n} a_i W^{q_i}-V\right\|_{\infty,Q_n}
\leq \frac1n. \end{equation*}
In particular, the sequence $(\sum_{i=1}^{r_n} a_i W^{q_i})_{n\in\mathbb{N}}$ is contained in $\mathscr{H}$ and converges uniformly to
$V$ on compact sets. As $\mathscr{H}$ is closed, it
follows that $V\in \mathscr{H}$. \qed
\end{lemma}

\begin{prop} For every integer $l\geq 0$, let $\mathscr{K}_l$ and
$\widehat{\mathscr{K}}_{l,q}$ be the convex cones defined
in~\eqref{setK} and~\eqref{eq:HatKl} respectively. Then 
\begin{equation}\label{malKrange}
\mathscr{K}_l=\{V\in \mathfrak{X}(Q)\mid V(q)\in \mathscr{K}_l(q)\}
\end{equation}
 and
\begin{equation}
\mathscr{K}_l(q)={\widehat{\mathscr{K}}_{l,q}},\qquad \mbox{for every $q\in Q$}.
\label{eq:propKHatK}
\end{equation}
 \label{prop:MalgrangeCone}
\end{prop}

\proof 
We first prove \eqref{malKrange} by induction on $l$. 
According to Lemma~\ref{lemma:H}, it is enough to prove that if $V\in\mathscr{K}_l$ and $a\in \mathcal{C}^\infty(Q,[0,\infty))$, then $aV\in \mathscr{K}_l$. The step $l=0$ is trivial. Let $l\geq 1$ and assume that 
the property is true for $l-1$.  
Since $\mathscr{K}_{l-1}-{\rm co}\left\{\langle Z
\colon Z \rangle \mid Z\in {\rm L}(\mathscr{K}_{l-1})\right\} $ is  ${\mathcal C}^0_{{\rm loc}}$-dense in $\mathscr{K}_l$ and 
$\{a^2\mid a\in \mathcal{C}^\infty(Q,\mathbb{R})\}$ is ${\mathcal C}^0_{{\rm loc}}$-dense in $\mathcal{C}^\infty(Q,[0,\infty))$, it is enough to prove that 
$a^2 V\in \mathscr{K}_l$ for every $V\in\mathscr{K}_{l-1}-{\rm co}\left\{\langle Z
\colon Z \rangle \mid Z\in {\rm L}(\mathscr{K}_{l-1})\right\} $ and $a\in \mathcal{C}^\infty(Q,\R)$.
Write $V=W-\sum_{j=1}^J \lambda_j\langle Z_j \colon Z_j\rangle$ with $W\in \mathscr{K}_{l-1}$, $\lambda_1,\dots,\lambda_J> 0$ with $\sum_{j=1}^J \lambda_j=1$ and $Z_1,\dots,Z_J\in {\rm L}(\mathscr{K}_{l-1})$. 
By induction hypothesis $a^2 W\in \mathscr{K}_{l-1}$. Moreover, ${\rm L}(\mathscr{K}_{l-1})$ is also a ${\mathcal C}^0_{{\rm loc}}$-closed set of $\mathfrak{X}(Q)$, 
closed with respect to finite linear combinations with coefficients in $\mathcal{C}^\infty(Q,[0,\infty))$. Hence, 
applying Lemma~\ref{lemma:H} to $\mathscr{H}={\rm L}(\mathscr{K}_{l-1})$ we deduce that $aZ_1,\dots,a Z_J$ belong to ${\rm L}(\mathscr{K}_{l-1})$. It can be easily proved using \eqref{triple_bracket} that 
$$\langle a Z_j\colon a Z_j\rangle= a^2 \langle Z_j\colon Z_j\rangle +b_j Z_j$$
for some smooth function $b_j$. By induction hypothesis, $b_j Z_j$ is in $\mathscr{K}_{l-1}$. Hence, $-a^2 \langle Z_j\colon Z_j\rangle$ lies in 
$\mathscr{K}_{l-1}-{\rm co}\left\{\langle Z
\colon Z \rangle \mid Z\in {\rm L}(\mathscr{K}_{l-1})\right\} $, concluding the proof of the identity 
$\mathscr{K}_l=\{V\in \mathfrak{X}(Q)\mid V(q)\in \mathscr{K}_l(q)\}$.

As a consequence, if $Z(q)\in
{\rm L}({\mathscr{K}}_{l}(q))$ for all $q\in Q$, then $Z\in L({\mathscr{K}}_{l})$, which implies that
\begin{equation}\label{LK}
{\rm L}(\mathscr{K}_{l}) =\{Z \in \mathfrak{X}(Q)\mid Z(q)\in
{\rm L}({\mathscr{K}}_{l}(q)) \quad \forall \; q\in Q\}.
\end{equation}

Let us now prove, again  by induction on $l$,
that 
 \eqref{eq:propKHatK} is true. 
The case  $l=0$ is trivial. Let us assume that 
\eqref{eq:propKHatK} 
 holds 
 for $l-1$,
and let us prove it for $l$. 
According to \eqref{LK} and the induction hypothesis, 
$$
{\rm L}(\mathscr{K}_{l-1}) =\{Z \in \mathfrak{X}(Q)\mid Z(q)\in
{\rm L}(\widehat{\mathscr{K}}_{l-1,q}) \quad \forall \; q\in Q\}.
$$
The definition of $\widehat{\mathscr{K}}_{l,q}$ then gives
$$\widehat{\mathscr{K}}_{l,q}=\overline{\mathscr{K}_{l-1}(q)-
{\rm
co}\left\{\langle Z \colon Z \rangle(q) \mid Z\in {\rm L}({\mathscr{K}}_{l-1})\right\}}$$
which gives the result when compared with 
\eqref{setK}. 
 \qed

\subsection{Kinematic reduction}\label{s-kinred}

It is already known in the literature that to perform certain motion planning  tasks it is useful to reduce a mechanical control system to a control-linear system 
in such a way that there exist relationships between the trajectories of both control systems. Before proceeding, we introduce some necessary definitions.

Let  $\mathscr{X}=\{X_1,\dots,X_m\}\subset \mathfrak{X}(Q)$ and consider the control-linear system  $(Q,\mathscr{X},\R^m)$ (defined in Section~\ref{control-linear}). 
Let us introduce the notations 
\begin{eqnarray*}
{\rm Sym}^{(0)}(\mathscr{Y})_q&=&{\rm span}_{\mathbb{R}}\mathscr{Y}(q),\\
{\rm Sym}^{(1)}(\mathscr{Y})_q&=&{\rm Sym}^{(0)}(\mathscr{Y})_q+{\rm span}_{\mathbb{R}}\{\langle W\colon Z \rangle (q) \mid W, Z \in \mathscr{Y}\}.
\end{eqnarray*}

\begin{defn}\label{Defn:KinRed} Let $\Sigma=(Q,\nabla,0, \mathscr{Y},\R^k)$ be a FACCS.
A driftless kinematic system $\Sigma_{\rm kin}=(Q,\mathscr{X},\mathbb{R}^{{m}})$ is a \textbf{kinematic reduction of $\mathbf{\Sigma}$} if
for every controlled trajectory $(\gamma,u_{\rm kin})$ of $\Sigma_{\rm kin}$ 
with $u_{\rm kin}$ smooth 
there exists $u$ smooth such that $(\gamma,u)$ is a controlled trajectory for $\Sigma$.
\end{defn}

Let us recall the following result from \cite{2005BulloAndrewBook}.

\begin{thm}[{\cite[Theorem 8.18]{2005BulloAndrewBook}}] \label{thm-KR}
Let $\Sigma$ and $\Sigma_{\rm kin}$ be as in Definition~\ref{Defn:KinRed}.
Assume that $\mathscr{X}$ and $\mathscr{Y}$ generate constant-rank distributions.
Then $\Sigma_{\rm kin}$ is a kinematic reduction of $\Sigma$ if and only if  ${\rm Sym}^{(1)}\mathscr{X}_q\subset \mathrm{span}_{\R}\mathscr{Y}(q)$ for every $q\in Q$.
\end{thm}

\subsection{A new criterion for trackability}\label{SNewTrack}

Let us now generalize 
the sufficient conditions
for tracking given in Proposition~\ref{prop:track-Z}.

\begin{thm}\label{Thm:ACCStrackZ}
Let $\Sigma=(Q,\nabla,Y, \mathscr{Y},\R^k)$ be a FACCS.
Define the families $\mathscr{Z}_i$, $i\in \N$, of
vector fields on $Q$ as in \eqref{setZ}.
Assume that
there exists $l\in \mathbb{N}$ such that 
\begin{enumerate}
\item $Y(t,p)\in \mathrm{span}_{\R}\mathscr{Z}_l(\tau_Q(p))$ for every $p\in TQ$;
\item the distributions $\mathrm{span}_{\R}\mathscr{Z}_{l-1}$, $\mathrm{span}_{\R}\mathscr{Z}_l$, and ${\rm Lie}\left(\mathscr{Z}_{l-1}\right)$ have constant rank; 
\item \label{Zbuona} for all $i\in\{0, \ldots, l-1\}$ and $Z\in \mathscr{Z}_i$,
$\langle Z\colon Z\rangle\in {\rm
span}_{\mathcal{C}^\infty(Q)}\mathscr{Z}_i$.
\end{enumerate}

Fix a smooth reference trajectory $\gref:I\rightarrow Q$ of class
${\mathcal C}^\infty$. 
If 
$\dgref(t)\in {\rm Lie}_{\gref(t)}\left(\mathscr{Z}_{l-1}\right)$
for every $t\in I$,  then $\gref$
is trackable. 
In particular, if 
${\rm Lie}
_q\left(\mathscr{Z}_{l-1}\right)=T_q Q$ for every $q\in Q$ then the \CTP\ holds.
\proof 
Let $l$ be as in the statement of the theorem and consider the FACCS 
$$\Sigma_l=(Q,\nabla,0, 
\mathscr{Z}_l,\mathbb{R}^{{m_l}}),$$
where $m_l$ is the cardinality of 
$\mathscr{Z}_l$.

As recalled in Theorem~\ref{thm-KR}, 
$$\Sigma_{l-1,\rm kin}=(Q,
\mathscr{Z}_{l-1},\mathbb{R}^{{m_{l-1}}})$$
is 
 a kinematic reduction of $\Sigma_l$,  
where $m_{l-1}$ is the cardinality of 
$\mathscr{Z}_{l-1}$, since ${\rm Sym}^{(1)}(\mathscr{Z}_{l-1})_q=\mathrm{span}_{\R}\mathscr{Z}_{l}(q)$ for every $q\in Q$.
Hence, every controlled trajectory 
of $\Sigma_{l-1,\rm kin}$ 
is also a 
controlled trajectory of $\Sigma_l$.

Since $\gref$ is tangent to the distribution $\mathrm{Lie}(\mathscr{Z}_{l-1})$, 
we deduce from Corollary~\ref{tracking-control-linear} that $\gref$ can be tracked with arbitrary precision by trajectories of $\Sigma_{l-1,\rm kin}$.

Hence, given
a positive tolerance $\epsilon$,  
there exists a 
controlled trajectory $(\gamma_1,u_1)$ of $\Sigma_{l-1,\rm kin}$ (still defined on the time-interval $I$) with $\gamma_1(0)=\gref(0)$,  $u_1:I\to \mathbb{R}^{m_{l-1}}$ smooth, and
such that 
$${\rm d} (\gref(t), \gamma_1(t))<\epsilon/2$$
for every time $t\in I$.

Now, by kinematic reduction,  
there exists $u_2:I\to \mathbb{R}^{m_l}$ smooth such that $(\gamma_1,u_2)$ is a controlled trajectory of $\Sigma_l$. 
Since the distribution generated by $\mathscr{Z}_l$ has constant rank, we can  represent $Y(t,\gamma_1(t))$ as a linear combination of $Z_1(\gamma_1(t)),\dots,Z_{m_l}(\gamma_1(t))$ with 
coefficients
depending smoothly on the time.  We recover that 
$$\nabla_{\dot\gamma_1(t)}\dot\gamma_1(t)-Y(t,\gamma_1(t))=\sum_{a=1}^{m_l} \lambda_a(t)Z_a(\gamma_1(t)),\qquad \forall t\in I, $$
with  $\lambda_1,\dots,\lambda_{m_l}\in C^{\infty}(I,\R)$.

Applying now Proposition~\ref{prop:track-Z}, we have that 
$\gamma_1$ is strongly trackable for $\Sigma$, and in particular there exists $u_3:I\to \mathbb{R}^k$ such that 
the trajectory $\gamma_3$ of $\Sigma$ corresponding to $u_3$ and with initial condition $\dot \gamma_3(0)=\dot\gamma_1(0)$ satisfies 
$${\rm d} (\gamma_1(t),\gamma_3(t))<\epsilon/2$$
 for every $t\in I$. 
We then conclude that $\gref$ is trackable for $\Sigma$.
\qed
\end{thm}

In order to generalize the argument to situations in which the hypothesis that
$\langle Z\colon Z\rangle\in {\rm
span}_{\mathcal{C}^\infty(Q)}\mathscr{Z}_i$ for every $Z\in \mathscr{Z}_i$ cannot be assumed, we 
introduce in the theorem below a new requirement on the linearity of the cones $\mathscr{K}_i$.

\begin{thm}\label{Thm:ACCStrackK}
Let $\Sigma=(Q,\nabla,Y, \mathscr{Y})$ be a FACCS.
Define the families $\K_i$, $i\in \N$, of
vector fields on $Q$ as in \eqref{setK}.
Assume that
there exists $l\in \mathbb{N}$ such that 
\begin{itemize}
\item $Y(t,p)\in \K_l(\tau_Q(p))$ for every $p\in TQ$;
\item for all $q\in Q$, ${\rm L}(\K_{l-1}(q))=\K_{l-1}(q)$ and ${\rm L}(\K_{l}(q))=\K_{l}(q)$;
\item  the distributions $\K_{l-1}$ and $\K_l$ and ${\rm Lie}\left(\K_{l-1}\right)$ 
have constant rank.
\end{itemize}

Fix a smooth reference trajectory $\gref:I\rightarrow Q$.
If  $\dgref(t)\in {\rm Lie}_{\gref(t)}\left(\K_{l-1}\right)$
for every $t\in I$,  then $\gref$
is trackable. 

In particular, if ${\rm Lie}_q\left(\K_{l-1}\right)=T_q Q$ for every $q\in Q$ then the \CTP\ holds.

\proof
The reasoning works similarly to the one used in the proof of Theorem~\ref{Thm:ACCStrackZ}. The first step is then to check that  ${\rm Sym}^{(1)}(\K_{l-1})\subseteq \K_l$, allowing  kinematic reduction arguments.
By definition of $\K_l$ (see~\eqref{setK}), we know that $-\langle Z\colon Z \rangle $ lies in $\K_l$ for each $Z\in {\rm L}(\K_{l-1})$. 
Moreover, we deduce from \eqref{malKrange} in Proposition~\ref{prop:MalgrangeCone} and the hypothesis ${\rm L}(\K_{j}(q))=\K_{j}(q)$ for $q$ in $Q$ and $j=l-1,l$, 
that ${\rm L}(\K_{j})=\K_{j}$ for $j=l-1,l$. 
Hence, $\pm \langle Z\colon Z \rangle $ lies in $\K_l$ for every $Z$ in $\K_{l-1}$. As the symmetric product of any vector field can be written as a linear combination of symmetric products of vector fields with themselves, 
\begin{equation*}
\langle Z\colon W \rangle=\dfrac{1}{2}\left( \langle Z+W\colon Z+W \rangle-\langle Z\colon Z \rangle-\langle W\colon W \rangle\right),
\end{equation*}
we conclude  that ${\rm Sym}^{(1)}(\K_{l-1})\subseteq \K_l$.

Let $\mathscr{V}=\{V_1,\dots, V_{m}\}$ be a set of generators of the distribution $q\mapsto \mathscr{K}_{l-1}(q)$ along $\gref$, i.e., 
$\mathscr{K}_{l-1}(\gref(t))=\mathrm{span}\{V_1(\gref(t)),\dots, V_{m}(\gref(t))\}$ for every $t\in I$.
It follows from Corollary~\ref{tracking-control-linear} that 
 the trajectories of $\Sigma_{l-1,\rm kin}=(Q,\mathscr{V},\mathbb{R}^{m})$ can track 
 $\gref$ with arbitrary precision. 
 Hence, given
a positive tolerance $\epsilon$,  
there exists a
controlled trajectory $(\gamma_1,u_1)$ of $\Sigma_{l-1,\rm kin}$ (still defined on the time-interval $I$) with $\gamma_1(0)=\gref(0)$ 
such that $u_1$ is smooth and 
$${\rm d} (\gref(t), \gamma_1(t))<\epsilon/2$$
for every time $t\in I$. 

Let $\mathscr{U}=\{U_1,\dots, U_{r}\}$ be a set of generators 
 of $\K_l$ along $\gamma_1$. 
Since 
$${\rm Sym}^{(1)}(\mathscr{V})_q\subseteq {\rm Sym}^{(1)}(\K_{l-1})_q\subseteq \K_l(q)=\mathrm{span}_{\R}\{U_1(q),\dots, U_{r}(q)\}$$
in a neighbourhood of the curve $\gref$,  we deduce from 
Theorem~\ref{thm-KR} that there exists $u_2:I\to \mathbb{R}^{r}$
such that $(\gamma_1,u_2)$ is a controlled trajectory of $\Sigma_l=(Q,\nabla, 0, \mathscr{U},\mathbb{R}^{r})$. 

Hence 
$$\nabla_{\dot\gamma_1(t)}\dot\gamma_1(t)=\sum_{a=1}^r \eta^+_a(t)U_a(\gamma_1(t))+\sum_{a=1}^r \eta^-_a(t)(-U_a(\gamma_1(t))),\qquad \forall \; t\in I, $$
with  $\eta^+_1,\dots,\eta^+_r,\eta^-_1,\dots,\eta^-_r\in \mathcal{C}^{\infty}(I,[0,+\infty))$.

Since $\K_l$ has constant rank, we can  represent $Y(t,\dot{\gamma}_1(t))$ as a linear combination of $U_1(\gamma_1(t))$, $\dots,U_r(\gamma_1(t))$ with coefficients
depending smoothly on the time.  We recover that 
$$\nabla_{\dot\gamma_1(t)}\dot\gamma_1(t)-Y(t,\dot{\gamma}_1(t))=\sum_{a=1}^r \lambda^+_a(t)U_a(\gamma_1(t))+\sum_{a=1}^r \lambda^-_a(t)(-U_a(\gamma_1(t))),\qquad \forall \; t\in I, $$
with  $\lambda^+_1,\dots,\lambda^+_r,\lambda^-_1,\dots,\lambda^-_r\in C^{\infty}(I,[0,+\infty))$.

Applying now Proposition~\ref{thm:oldtrackK}, we have that 
$\gamma_1$ is strongly trackable for $\Sigma$, and in particular there exists $u_3:I\to \mathbb{R}^k$ such that 
the trajectory $\gamma_3$ of $\Sigma$ corresponding to $u_3$ and with initial condition $\dot \gamma_3(0)=\dot\gamma_1(0)$ satisfies 
$${\rm d} (\gamma_1(t),\gamma_3(t))<\epsilon/2$$
 for every $t\in I$, and 
 we conclude that $\gref$ is trackable for $\Sigma$.
\qed

\end{thm}

\subsection{Example}\label{SExample}

In this section we apply Theorem~\ref{Thm:ACCStrackZ} to a control system studied in~\cite{BS2010} and \cite{Mario}, completing the discussion on its trackability by 
tackling a case which was not covered by previously known criteria. 

The system models a neutrally buoyant ellipsoidal vehicle immersed in a
infinite volume fluid that is inviscid, incompressible and whose
motion is irrotational. The dynamics are obtained through Kirchhoff
equations \cite{Lamb} and have a particularly simple form due to
some symmetry assumption on the distribution of mass (see
\cite{Mario} for details and also \cite{munnier}
for general overview of control 
motion in a potential
fluid).

Consider the coordinates $(\omega,v)$ for the angular and linear
velocity of the ellipsoid with respect to a body-fixed
coordinate frame. Then the impulse $(\Pi,P)$ of the system is given
by
\[ \begin{pmatrix} \Pi \\ P \end{pmatrix} = {\mathcal M}
\begin{pmatrix} \omega \\ v \end{pmatrix}  \]
where, under the symmetry assumptions mentioned above,
$$
{\mathcal M}=\mathrm{diag}(J_1,J_2,J_3,M_1,M_2,M_3),
$$
$\mathrm{diag}(J_1,J_2,J_3)$ is the usual inertia matrix, and $M_1,M_2,M_3$ take into account the mass of the submarine and the added masses due to the action of the fluid.

The configuration manifold $Q$ for this problem is the Special
Euclidean group or the group of rigid motions ${\rm SE}(3)$, which is
homeomorphic to $SO(3)\times \mathbb{R}^3$. Let $(A,r)\in {\rm SE}(3)$ be
the attitude and the position of the ellipsoid.
Denote by $S\colon \mathbb{R}^3\rightarrow \mathfrak{so}(3)$ the
linear bijection between $\mathbb{R}^3$ and the linear algebra
$\mathfrak{so}(3)$ of $SO(3)$ such that
\[S(x_1,x_2,x_3)=\begin{pmatrix}0 & -x_3 & x_2 \\ x_3 & 0 & -x_1 \\ -x_2 & x_1 & 0
\end{pmatrix}.\]
 The dynamics
of the controlled system are given by
\begin{equation}\label{eq-below}
\ds{\frac{{\rm d}A}{{\rm d} t}=
A S(\omega), \quad \frac{{\rm d}r}{{\rm d} t}= Av,}\end{equation}
and
\begin{equation}\label{eq-KirchSubm}
\ds{\frac{{\rm d} \Pi}{{\rm d}t}} = \Pi \times \omega + P \times v+
\begin{pmatrix} u_1 \\ u_2 \\ 0 \end{pmatrix}, \quad
\ds{\frac{{\rm d} P}{{\rm d}t}}  = P \times \omega+
\begin{pmatrix} 0 \\ 0 \\ u_3 \end{pmatrix}.\end{equation}
The controls correspond to a linear acceleration along one of the three axes of the submarine and to two angular accelerations around the other two axes.
It was proven in \cite{Mario} that 
if  $M_1\ne M_2$ then system~\eqref{eq-below}--\eqref{eq-KirchSubm} satisfies the SCTP.
In  \cite{BS2010}, moreover, based on general quantitative estimates of the convergence yielding the sufficient conditions for tracking recalled in 
Proposition~\ref{prop:track-Z}, an explicit tracking algorithm was proposed. 

The Lie group structure of the configuration manifold can be exploited to compute Lie brackets and symmetric products of left-invariant vector fields, and in particular of the control  vector fields. (Otherwise, one can directly apply \eqref{triple_bracket}.)
It turns out (see~\cite{2005BulloAndrewBook,Bullo} for details)
 that for $\eta_1=\begin{pmatrix} S(w_1) & v_1 \\ 0 & 0 \end{pmatrix}$, $\eta_2=\begin{pmatrix} S(w_2) & v_2 \\ 0 & 0 \end{pmatrix}\in \mathfrak{se}(3)$, identified with the corresponding left-invariant vector fields,  
\begin{equation}
\left[\eta_1,\eta_2\right]= 
\begin{pmatrix} [S(w_1),S(w_2)] & S(w_1) v_2-S(w_2)v_1 \\ 0 & 0 \end{pmatrix}.\label{eq:LieBracket}
\end{equation}
Let $\{e_1,\dots, e_6\}$ be a basis adapted to the coordinates $(A,r)$ so that $\eta_a=\eta_a^ie_i$ for $a=1,2$. It is then possible to compute the symmetric product as follows:
\begin{equation}
\langle \eta_1 \colon \eta_2 \rangle=- {\mathcal M}^{-1} ({\rm ad}_{\eta_1}^* {\mathcal M} \eta_2+{\rm ad}_{\eta_2}^* {\mathcal M} \eta_1),
\label{sym-sub}
\end{equation}
where $({\rm ad}^*_\eta \alpha) \xi=\alpha({\rm ad}_\eta \xi)$ for $\eta,\xi \in \mathfrak{se}(3)$, $\alpha\in \mathfrak{se}^*(3)$.

The structural constants with respect to the basis $\{e_1,\dots, e_6\}$ are defined as
\begin{equation*}
c_{ij}^k=[e_i,e_j]^k, \quad \gamma_{ij}^k=\langle e_i\colon e_j \rangle^k.
\end{equation*}
By the expressions of the Lie bracket and the symmetric product given in~\eqref{eq:LieBracket} and \eqref{sym-sub}, respectively, it follows that
 $$\gamma^k_{ij}=-{\mathcal M}^{hk}({\mathcal M}_{il} c^l_{jh}+{\mathcal M}_{jl} c^l_{ih}),$$  being ${\mathcal M}^{hk}$ the entries of the inverse matrix of ${\mathcal M}$.

One can easily compute that 
\begin{equation*}
\begin{array}{l}
c^1_{23}=c^2_{31}=c^3_{12}=c^4_{26}=c^4_{53}=c^5_{34}=c^5_{61}=c^6_{15}=c^6_{42}=1,\\  
c^1_{32}=c^2_{13}=c^3_{21}=c^4_{62}=c^4_{35}=c^5_{43}=c^5_{16}=c^6_{51}=c^6_{24}=-1,
\end{array}
\end{equation*}
and
\begin{equation*}
\begin{array}{lcl}
\gamma^1_{32}=\gamma^1_{23}=\dfrac{J_3-J_2}{J_1}, && \gamma^1_{56}=\gamma^1_{65}=\dfrac{M_3-M_2}{J_1}, \\ [3mm]
\gamma^2_{31}=\gamma^2_{13}=\dfrac{J_1-J_3}{J_2}, && \gamma^2_{46}=\gamma^2_{64}=\dfrac{M_1-M_3}{J_2}, \\ [3mm]
  \gamma^3_{21}=\gamma^3_{12}=\dfrac{J_2-J_1}{J_3}, &&   \gamma^3_{45}=\gamma^3_{54}=\dfrac{M_2-M_1}{J_3}, \\[3mm]
  \gamma^4_{26}=\gamma^4_{62}=\dfrac{M_3}{M_1}, &&   \gamma^4_{35}=\gamma^4_{53}=-\dfrac{M_2}{M_1}, \\[3mm]
  \gamma^5_{16}=\gamma^5_{61}=-\dfrac{M_3}{M_2}, &&   \gamma^5_{34}=\gamma^5_{43}=\dfrac{M_1}{M_2}, \\[3mm]
  \gamma^6_{15}=\gamma^6_{51}=\dfrac{M_2}{M_3}, &&   \gamma^6_{24}=\gamma^6_{42}=-\dfrac{M_1}{M_3}, 
\end{array}
\end{equation*}
while all other structural constants are equal to zero.

Notice that 
the control vector fields are the vertical lift to $T{\rm SE}(3)$ of 
\begin{equation*}
Y_1=\dfrac{1}{J_1}\, e_1, \quad Y_2=\dfrac{1}{J_2} \, e_2, \quad Y_3=\dfrac{1}{M_3} \, e_6,
\end{equation*}
where left-invariant vector fields are identified with elements of the Lie algebra of ${\rm SE}(3)$.

We focus here on the case that was left unanswered in~\cite{BS2010} and \cite{Mario}, namely, the case where
\begin{equation}\label{gamma=0}
J_1=J_2,\qquad M_1=M_2.
\end{equation}
Such a case cannot be studied using the most general sufficient conditions for trackability given in \cite{BS2010} (recalled in Proposition~\ref{thm:oldtrackK}), as illustrated by the computations here below.

Under condition \eqref{gamma=0}, one easily computes that
\begin{equation*}
\langle Y_1 \colon Y_2 \rangle=0,\quad \langle Y_1 \colon Y_3 \rangle=-\dfrac{1}{J_1M_1} \, e_5, \quad  \langle Y_2 \colon Y_3 \rangle=\dfrac{1}{J_1M_1} \, e_4.
\end{equation*}
Moreover,
$$
\langle e_j\colon e_j\rangle=0\quad\mbox{ for $1\leq j\leq 6.$}
$$
Let $Y_4=-\dfrac{1}{J_1M_1} \, e_5$ and $Y_5=\dfrac{1}{J_1M_1} \, e_4$. 
Hence, $\mathscr{Z}_1=\{Y_j\mid 1\leq j\leq 5\}$ and condition (\ref{Zbuona}) in the statement of Theorem~\ref{Thm:ACCStrackZ} 
is satisfied for $i=0$. 

Straightforward computations also give that $\langle Y_1\colon Y_4\rangle$ and 
$\langle Y_2\colon Y_5\rangle$
are proportional to $Y_3$, $\langle Y_3\colon Y_4\rangle$ is proportional to $Y_1$, $\langle Y_3\colon Y_5\rangle$ is proportional to $Y_5$, while 
$$\langle Y_1\colon Y_5\rangle=\langle Y_2\colon Y_4\rangle=\langle Y_4\colon Y_5\rangle=0.$$

Thus,  $\mathscr{Z}_i(q)={\rm span}_{\mathbb{R}}\{ e_1,e_2,e_4,e_5,e_6\}_q$ for every  $q\in {\rm SE}(3)$ and every $i\geq 1$, with condition (\ref{Zbuona}) in the statement of Theorem~\ref{Thm:ACCStrackZ} satisfied for every $i\geq 0$.

Note that 
\begin{equation*}
[Y_1,Y_2]=\dfrac{1}{J_1^2}e_3\,.
\end{equation*}
Thus, ${\rm Lie}_q\left(\mathscr{Z}_1\right)=T_qQ$ for every $q\in {\rm SE}(3)$. By Theorem~\ref{Thm:ACCStrackZ} the configuration tracking property is guaranteed for these control vector fields. 
This completes the results in \cite{BS2010} and \cite{Mario}, allowing to conclude that system~\eqref{eq-below}--\eqref{eq-KirchSubm}
satisfies 
the CTP for any choice of the (positive definite) diagonal inertial matrix ${\mathcal M}$.

\section*{Acknowledgements}

The authors would like to thank Gr\'egoire Charlot, whose help was crucial for obtaining the results in Section~\ref{s-gregoire}.

\bibliographystyle{elsarticle-num}
\bibliography{refs}

\begin{thebibliography}{10}
\expandafter\ifx\csname url\endcsname\relax
  \def\url#1{\texttt{#1}}\fi
\expandafter\ifx\csname urlprefix\endcsname\relax\def\urlprefix{URL }\fi
\expandafter\ifx\csname href\endcsname\relax
  \def\href#1#2{#2} \def\path#1{#1}\fi

\bibitem{2005BulloAndrewBook}
F.~Bullo, A.~D. Lewis, Geometric control of mechanical systems, Vol.~49 of
  Texts in Applied Mathematics, Springer-Verlag, New York, 2005, modeling,
  analysis, and design for simple mechanical control systems.

\bibitem{BS2010}
M.~Barbero-Li{\~n}{\'a}n, M.~Sigalotti, High-order sufficient conditions for
  configuration tracking of affine connection control systems, Systems Control
  Lett. 59~(8) (2010) 491--503.

\bibitem{Mario}
T.~Chambrion, M.~Sigalotti, Tracking control for an ellipsoidal submarine
  driven by {K}irchhoff's laws, IEEE Trans. Automat. Control 53~(1) (2008)
  339--349.

\bibitem{2008Texas}
M.~Barbero-Li{\~n}{\'a}n, M.~C. Mu{\~n}oz-Lecanda,
  \href{http://dx.doi.org.proxy.queensu.ca/10.3934/dcdss.2010.3.1}{Strict
  abnormal extremals in nonholonomic and kinematic control systems}, Discrete
  Contin. Dyn. Syst. Ser. S 3~(1) (2010) 1--17.
\newblock \href {http://dx.doi.org/10.3934/dcdss.2010.3.1}
  {\path{doi:10.3934/dcdss.2010.3.1}}.
\newline\urlprefix\url{http://dx.doi.org.proxy.queensu.ca/10.3934/dcdss.2010.3.1}

\bibitem{Bressan}
A.~Bressan, Z.~Wang, \href{http://dx.doi.org/10.1016/j.jde.2009.01.014}{On the
  controllability of {L}agrangian systems by active constraints}, J.
  Differential Equations 247~(2) (2009) 543--563.
\newblock \href {http://dx.doi.org/10.1016/j.jde.2009.01.014}
  {\path{doi:10.1016/j.jde.2009.01.014}}.
\newline\urlprefix\url{http://dx.doi.org/10.1016/j.jde.2009.01.014}

\bibitem{Jorge}
S.~Mart{\'{\i}}nez, J.~Cort{\'e}s,
  \href{http://dx.doi.org/10.1023/A:1023275502000}{Motion control algorithms
  for simple mechanical systems with symmetry}, Acta Appl. Math. 76~(3) (2003)
  221--264.
\newblock \href {http://dx.doi.org/10.1023/A:1023275502000}
  {\path{doi:10.1023/A:1023275502000}}.
\newline\urlprefix\url{http://dx.doi.org/10.1023/A:1023275502000}

\bibitem{AgSa1}
A.~A. Agrachev, A.~V. Sarychev,
  \href{http://dx.doi.org/10.1007/s00021-004-0110-1}{Navier-{S}tokes equations:
  controllability by means of low modes forcing}, J. Math. Fluid Mech. 7~(1)
  (2005) 108--152.
\newblock \href {http://dx.doi.org/10.1007/s00021-004-0110-1}
  {\path{doi:10.1007/s00021-004-0110-1}}.
\newline\urlprefix\url{http://dx.doi.org/10.1007/s00021-004-0110-1}

\bibitem{AgSa2}
A.~A. Agrachev, A.~V. Sarychev,
  \href{http://dx.doi.org/10.1007/s00220-006-0002-8}{Controllability of 2{D}
  {E}uler and {N}avier-{S}tokes equations by degenerate forcing}, Comm. Math.
  Phys. 265~(3) (2006) 673--697.
\newblock \href {http://dx.doi.org/10.1007/s00220-006-0002-8}
  {\path{doi:10.1007/s00220-006-0002-8}}.
\newline\urlprefix\url{http://dx.doi.org/10.1007/s00220-006-0002-8}

\bibitem{Malgrange67}
B.~Malgrange, Ideals of differentiable functions, Tata Institute of Fundamental
  Research Studies in Mathematics, No. 3, Tata Institute of Fundamental
  Research, Bombay; Oxford University Press, London, 1967.

\bibitem{AbrahamMarsden}
R.~Abraham, J.~E. Marsden, Foundations of mechanics, Benjamin/Cummings
  Publishing Co., Inc., Advanced Book Program, Reading, Mass., 1978, second
  edition, revised and enlarged, With the assistance of Tudor Ra{\c{t}}iu and
  Richard Cushman.

\bibitem{Liu}
W.~Liu, \href{http://dx.doi.org/10.1137/S0363012993260501}{An approximation
  algorithm for nonholonomic systems}, SIAM J. Control Optim. 35~(4) (1997)
  1328--1365.
\newblock \href {http://dx.doi.org/10.1137/S0363012993260501}
  {\path{doi:10.1137/S0363012993260501}}.
\newline\urlprefix\url{http://dx.doi.org/10.1137/S0363012993260501}

\bibitem{LS}
H.~J. Sussmann, W.~Liu, Limits of highly oscillatory controls and approximation
  of general paths by admissible trajectories, in: Proceedings of the 30th IEEE
  Conference on Decision and Control, Brighton, UK, 1991.

\bibitem{Jakubczyk}
B.~Jakubczyk,
  \href{http://users.ictp.trieste.it/$\sim$pub\underline{\hspace{2mm}}off/lectures/vol8.html}{Introduction
  to geometric nonlinear control, controllability and lie bracket}, in:
  Mathematical Control Theory, Vol.~8, ICTP Lecture Note Series, 2002.
\newline\urlprefix\url{http://users.ictp.trieste.it/$\sim$pub\underline{\hspace{2mm}}off/lectures/vol8.html}

\bibitem{Lamb}
H.~Lamb, Hydrodynamics, sixth Edition, Cambridge Mathematical Library,
  Cambridge University Press, Cambridge, 1993.

\bibitem{munnier}
A.~Munnier, \href{http://dx.doi.org/10.1007/s00332-009-9047-0}{Locomotion of
  deformable bodies in an ideal fluid: {N}ewtonian versus {L}agrangian
  formalisms}, J. Nonlinear Sci. 19~(6) (2009) 665--715.
\newblock \href {http://dx.doi.org/10.1007/s00332-009-9047-0}
  {\path{doi:10.1007/s00332-009-9047-0}}.
\newline\urlprefix\url{http://dx.doi.org/10.1007/s00332-009-9047-0}

\bibitem{Bullo}
F.~Bullo, Invariant affine connections and controllability on {L}ie groups,
  technical Report for Geometric Mechanics, California Institute of Technology,
  \rm{http://www.cds.caltech.edu/~marsden/wiki/uploads/projects/geomech/Bullo1995.pdf}
  (1995).

\end{thebibliography}

\end{document}